\documentclass{amsart}
\usepackage{amsfonts}
\usepackage{amsmath}
\usepackage{amssymb}
\usepackage[applemac]{inputenc}
\usepackage{url}
\usepackage{hyperref}
\usepackage{color}

\newtheorem{theorem}{{Theorem}}[section]
\newtheorem{proposition}[theorem]{{Proposition}}
\newtheorem{isom.ext}[theorem]{{Trivial isometric extension}}

\newtheorem{lemma}[theorem]{{Lemma}}

\newtheorem{remark}[theorem]{{Remark}}

\newtheorem{example}[theorem]{{Example}}

 \newtheorem{definition}{{Definition}}

\definecolor{greenbf}{rgb}{0, 0.7 ,0.3}
 \hypersetup{
colorlinks=true, breaklinks=true, urlcolor= blue, linkcolor= red,
citecolor= {greenbf}, }

\def\R{\mathbb{R}}

\def\GL{{\sf{GL}}}

\begin{document}
\title{On Rigidity of Generalized Conformal Structures}
\author{Samir Bekkara}
\address{S. Bekkara, Department of Mathematics, UST-Oran, Algeria }
\email{samir.bekkara@gmail.com}
\author{ Abdelghani Zeghib}
\address{A. Zeghib, CNRS, UMPA, ENS-Lyon, France }
\email{abdelghani.zeghib@ens-lyon.fr \hfill\break\indent
 \url{http://www.umpa.ens-lyon.fr/~zeghib/}}
\date{\today}

\begin{abstract}
The classical Liouville Theorem on conformal transformations determines
local conformal transformations on the Euclidean space of dimension $\geq 3$%
. Its natural adaptation to the general framework of Riemannian structures
is the 2-rigidity of conformal transformations, that is such a
transformation is fully determined by its 2-jet at any point. We prove here
a similar rigidity for generalized conformal structures defined by giving a
one parameter family of metrics (instead of scalar multiples of a given one)
on each tangent space.
\end{abstract}

\maketitle
\tableofcontents

\section{Introduction}

\baselineskip 0.5 cm

\subsubsection*{Rough notion}

For a vector space $E$, let $\mathsf{Sym} (E)$ be the space of symmetric
bilinear forms on $E$, $\mathsf{Sym} ^+(E)$ those which are positive
definite, and $\mathsf{Sym} ^*(E)$ the non-degenerate ones.

For a manifold $M$, one defines similarly fiber bundles $\mathsf{Sym} (TM)$,
$\mathsf{Sym} ^+(TM)$ and $\mathsf{Sym} ^*(TM)$ associated to its tangent
bundle $TM$.

A Riemannian metric is nothing but a section of $\mathsf{Sym} ^+(TM)$.
Recall on the other hand that a (Riemannian) conformal structure consists in
giving a class $[g]$ of Riemannian metrics, for the conformal equivalence
relation $\sim {}$ between metrics: $g_1 \sim g_2$ if there exists a
function $\sigma$ on $M$ such that $g_1 = e^\sigma g_2$. Thus, a conformal
structure consists in giving a section of the projectivized of $\mathsf{Sym}
^+(TM)$.

Equivalently, a conformal structure consists in giving for each point $x \in
M$, a half line in $\mathsf{Sym} ^+(T_xM)$.

We are now going to introduce a first rough definition of
generalized conformal structures (GCS for short) by associating to
each $x \in M $ a (non-parameterized) curve in $\mathsf{Sym}
^+(T_xM)$. Say, this consists in giving a subset $\mathcal{C}
\subset \mathsf{Sym} ^+(TM)$ such that the fibers of the projection
$\mathcal{C} \to M$ have dimension $\leq 1$
and are non-empty. One naturally defines the image of such a structure $%
\mathcal{C}$ by a diffeomorphism, and an automorphism group $Aut(\mathcal{C}%
) $. In the sequel, automorphisms will be alternatively called isometries.

\bigskip

Our goal is to study such objects from the point of view of being ``rigid
geometric structures''. Roughly speaking, $d$-rigidity means that an
automorphism is fully determined by its jet up to order $d$ at any point. We
have here two ``limit'' cases, that where the $\mathcal{C}$-fibers are points
(a Riemannian metric), and the other where the $\mathcal{C}$-fibers are
half-lines (a conformal structure). It is known that Riemannian metrics are
1-rigid, whereas conformal structures are 2-rigid in dimension $\geq 3$;
this is the essence of classical Liouville Theorem. Our generalized case
here when the $\mathcal{C}$-fibers are general curves may be expected to be
as rigid as the conformal case, that is one has 2-rigidity. In some sense,
one naturally expects that when going from straight lines to general curves,
one can not lose of rigidity because one gets more constraints on isometries.

\subsubsection{A First example}

\label{first} Let us start by this general example which will give evidence
that some topological tameness hypotheses on $\mathcal{C}$ are in order. Let
$\phi^t$ be a flow on $M$ and $g_0$ any initial metric on $M$. For any $x $,
give $\mathcal{C}_x$ as the (parameterized) curve $t \to (\phi^t_*g_0)_x \in
\mathsf{Sym} ^+(T_xM)$, here $\phi^t_* g_0$ is the image of $g_0$ by $\phi^t$%
.

Observe that $\phi^t \in Aut(\mathcal{C})$. Thus, any flow gives rise to a
rough GCS with a non-trivial automorphism group, which may have a strong
dynamics. One can not expect for such a structure to behave as a nice
geometric structure!

\subsubsection{(Regular) Definition}

\label{definition} We are now going to propose a definition of GCS which
will be proved to be adapted to our rigidity hope, just by assuming that the
corresponding subset $\mathcal{C}$ is a manifold. \bigskip

\textit{More precisely, let us say $\mathcal{C}$ is a regular GCS if $%
\mathcal{C}$ is a submanifold of dimension $\dim M +1$ in $\mathsf{Sym}
^+(TM)$, which is transverse to the fibers (of $\mathsf{Sym} ^+(TM) \to M$).}
Equivalently, the projection $\mathcal{C} \to M$ is a submersion and $\dim
\mathcal{C} = \dim M + 1$.

\medskip
Each fiber $\mathcal{C}_x$ is thus a (non-necessarily connected)
embedded
1-dimensional submanifold.
In the case of a classical conformal structure, $\mathcal{C}$ is in fact a
closed submanifold and it fibers over $M$.

Let us say that $\mathcal{C}$ is \textbf{generic} if the tangent direction
of $\mathcal{C}_x$ at any of its points belongs to $\mathsf{Sym} ^*(T_xM)$.
In other words, if $\mathcal{C}_x$ is parameterized as a curve $t \in
\mathbb{R} \to c_x(t) \in \mathsf{Sym} (T_xM)$, then $c_x^\prime(t)$ is
assumed to be non-degenerate. For example, classical conformal structures
are generic.

\subsubsection{A second example, Infinitesimally Homogeneous case}

(see \ref{H.structure}). \label{second} Let us consider the situation where
there is a 1-dimensional submanifold $\mathcal{C}_0 \subset \mathsf{Sym} ^+(%
\mathbb{R}^n)$ such that for any $x$, $%
\mathcal{C}_x = A_x^*(\mathcal{C}_0)$ where $A_x: \mathbb{R}^n \to T_x M$
is a linear isomorphism and $A_x^*$ is the associated map $\mathsf{Sym} ^+(%
\mathbb{R}^n) \to \mathsf{Sym} ^+ (T_xM)$. If the dependence $x \to A_x$ is
smooth,  then $\mathcal{C}$ is a GCS, which as in the
standard conformal case, gives rise to a fibration $\mathcal{C} \to M$.

Let us here mention one useful and beautiful property of this moduli space $%
\mathsf{Sym} ^+(\mathbb{R}^n)$, or more generally any $\mathsf{Sym} ^+(E)$,
for $E$ a linear space; this is the space of ``linear'' Riemannian metrics
on $E$, and it admits itself a canonical Riemannian metric, which makes it
as a universal symmetric space under the natural action of $\mathsf{GL}(E)$
(see §\ref{metric.on.metrics}).

Let    $H$ be the  stabilizer subgroup in $\mathsf{GL}(n,
\mathbb{R})$ of $\mathcal{C}_0$.
For any $x \in M$, consider $I_x$ the set of isomorphisms $T_xM \to \mathbb{R%
}^n$ sending $\mathcal{C}_x$ to $\mathcal{C}_0$. This is clearly an $H$%
-orbit in the $\mathsf{GL}(n, \mathbb{R})$-space $\mathsf{{Isom}}(T_xM,
\mathbb{R}^n)$, that is the fiber over $x$ of the frame bundle $P_M \to M$.
When $x $ runs over $M$, we therefore get a section of $P_M /H \to M$, that
is an $H$-structure on $M$.

Conversely, an $H$-structure gives naturally a GCS of type $\mathcal{C}_0$.
Indeed, by definition of an $H$-structure, it consists in giving for any $x$%
, an $H$-orbit $I_x$ as above. The pull back $\mathcal{C}_x$ of the curve $%
\mathcal{C}_0$ by any element of $I_x$ does not depend on the choice of such
element.

\subsubsection{Rigidity}

\label{Rigidity}

Let $\phi$ be a diffeomorphism of $M$ and $\phi^*$ its induced action on $%
\mathsf{Sym} (TM)$. Then $\phi$ is an automorphism of $\mathcal{C}$ (a GCS
on $M$) if $\phi^*(\mathcal{C}) = \mathcal{C}$.

The following discussion applies to diffeomorphisms sending a point $p \in M$
to another $q \in M$, but we will be specially interested in the case $p = q$.
 Then, define $\phi$ to be isometric
up to order 1 at $p$, if $\phi(p) = p$ and $(\phi^*(\mathcal{C}))_p =
\mathcal{C}_p$, i.e. $\phi^*(\mathcal{C})$ and $\mathcal{C}$ meet along $%
\mathcal{C}_p$. We say that $\phi$ is isometric up to order $d \geq
1$ at $p $ or simply a $d$-isometry if $\phi^*(\mathcal{C})$ and
$\mathcal{C}$ have contact of order $(d-1)$ along $\mathcal{C}_p$.
In order to be complete, let us precise that the local model of two
$k$-submanifolds $V$ and $W$ of $\mathbb{R}^N$ having a contact at
order $s$ along a curve $\mathcal C_0$, is that where $V =
\mathbb{R}^k$,  $ \mathcal C_0 = \R \subset \R^k$, and $W$ is the
graph of a function $f: \mathbb{R}^k \to \mathbb{R}^{N-k}$ having a
vanishing Taylor expansion up to order $s$ at all points of
$\mathcal C_0$. Let us also indicate that we will say that $\phi$
has a trivial $d$-jet at $p$ if it has the same $d$-jet as the
identity at $p$.

Rigidity at order 2 of classical conformal structures in dimension $\geq 3$,
is essentially equivalent to the classical Liouville Theorem stating that
any local conformal transformation of a Euclidean space of dimension $%
\geq 3$, is a composition of a translation, a similarity and an inversion
(see for instance \cite{berger, Spivak} and \cite{Frances}). There are many
approaches to this rigidity, including that by the theory of $H$-structures
of finite type, via computation of the prolongation spaces for the conformal
group $H=\mathbb{R}.\mathsf{O}(n)$, see \cite{Kobaya, sternberg, gromov 1,
Ballmann}. Here we generalize to generic GCS:

\begin{theorem}[Generalized Liouville Theorem]
\label{Liouville} Let $\mathcal{C}$ be a generic generalized conformal
structure on a manifold of dimension $\geq 3$.
 Then $\mathcal{C}$ is $d$-rigid, for any $d \geq 2$,  that is a $(d+1)$-isometry at a given point with a
trivial $d$-jet, has a trivial $(d+1)$-jet.
\end{theorem}

An alternative formulation would be that if a $(d+1)$-isometry at some point $p$, with $d \geq 2$,  has a trivial 2-jet   then it has a
trivial $(d+1)$-jet. In particular, if a smooth local isometry has a trivial 2-jet at
$p$, then it has  trivial infinite jet, i.e. it is infinitely tangent to  the identity at $p$.

In a first version of the present   article we just proved 2-rigidity,  we then investigate
 the general case   after  request of the referee.
  In fact, in the case of geometric structures in the Gromov  sense, it is a general fact that
 $k$-rigidity implies $d$-rigidity for any $d \geq k$, and that a local isometry
 with a trivial $k$-jet at some point is the identity in a neighbourhood of it.  These
 implications are somehow ``tautological'' but follow from a highly
 sophisticated machinery.
 Adaptation of this formalism to our situation seems
 possible but needs a specific and independent  investigation,
 see  \S \ref{geometric_structure?} and \S \ref{A-type?} for a preliminary  discussion
 on these aspects.
 Actually, our proof of
 $d$-rigidity of GCS for $d >2$ is done  by  rather adapting
  computations of the case $d=2$.

One essential motivation behind the Gromov
notion of $k$-rigidity for a geometric structure is that it gives a
way to prove that its  isometry group is of Lie type.
We will  give details
 about this question of  Lie group structure   of
isometry groups of GCS (as well as  lightlike metrics) in a forthcoming
article \cite{BZ_Next}.

\begin{example}[A non rigid example]
Consider canonical coordinates $(x^1, \ldots, x^n)$ on $\mathbb{R}^n$. Endow
it with $\mathcal{C}$ the ``constant''  GCS given by the curve of Euclidean
metrics $t(dx^1)^2 + (dx^2)^2 + \ldots + (dx^n)^2 $, $t>0$. This $\mathcal{C%
}$ is in fact an $H$-structure. Any diffeomorphism $\phi$ of the form $%
\phi(x^1, \ldots, x^n) = (f(x^1), x^2, \ldots, x^n)$ is isometric. This
structure is not rigid, indeed $\mathcal{C}$ is not generic.

Note however that it may happen for a GCS to be rigid, even if it is not
generic (such a situation is thus not covered by our result). For instance, for
an $H$-structure with $H$ a one parameter subgroup of $\mathsf{GL}(n,
\mathbb{R})$, one can prove it  has finite type iff the Lie subalgebra of $H$ contains no
matrices of rank 1, in which case the structure has finite type 1, i.e. it
is 1-rigid like a Riemannian metric (see  for instance  (\cite{Kobaya}, page 4) for one implication).
\end{example}

\begin{remark}
More generalizations of conformal structures can be obtained by relaxing the
dimension condition on $\mathcal{C}$, say by assuming $\dim \mathcal{C} =
\dim M +l$, where $l$ may be bigger than 1. The rigidity discussion will
then depend on $l$ and $\dim M$?
\end{remark}

\section{Further investigations}

\subsection{Interplay with Lightlike metrics}

\label{transversalyRiemaninnian} Our motivation behind the study of GC
structures was in fact their relation with the lightlike ones that we
considered in \cite{BZ}. Recall that a \textbf{lightlike} metric $g$ on a
manifold $\mathcal{V}$ is a tensor which is a positive non-definite
quadratic form of 1-dimensional kernel in each tangent space of $\mathcal{V}$
\cite{BFZ}. The kernel of $g$ is a direction field $N$, tangent to a
1-dimensional foliation $\mathcal{N}$ called null or characteristic.

This null foliation is not necessarily oriented by a (global) non-singular
vector field $X$ tangent to it, but we can assume it is the case by passing
to a double cover, or arguing locally. Then, the lightlike structure is said
to be \textbf{transversally Riemannian} if the Lie derivative $L_Xg = 0$.
Let us say $g$ is \textbf{nowhere} transversally Riemannian, if $L_Xg(x)
\neq 0$, for any $x$. In the stronger situation where $L_Xg$ is
non-degenerate on $T\mathcal{V}/N$, $g$ is said to be \textbf{generic}. Both
this genericity condition or being transversally Riemannian are independent
of the choice of a particular $X$ orienting $\mathcal{N}$.

\subsubsection{From GCS to lightlike structures}

Let $\mathcal{C} \subset \mathsf{Sym} ^+(TM)$ be a GCS on $M$ and $\pi:
\mathcal{C} \to M$ the projection. Let $x \in M$, $q \in \mathcal{C}_x =
\pi^{-1}(x)$, and consider the projection $d_q \pi: T_q \mathcal{C} \to T_x
M $. Now, let $q$ play the role of a (definite) scalar product on $T_xM$,
its pull back by $d_q\pi$ is a lightlike scalar product on $T_q \mathcal{C}$%
. We get in this way a tautological lightlike metric on $\mathcal{C}$.

Observe that this lightlike metric on $\mathcal{C}$ is nowhere transversally
Riemannian, and also that $\mathcal{C} $ is generic as a GCS iff its
lightlike metric is generic (as defined previously). To see all this, one
writes all things in a local chart. If $x= (x^1, \ldots, x^n) $ are local
coordinates on $M$, then $\mathcal{C}$ admits a parameterization $(t, x) \to
c(t, x) \in \mathsf{Sym} ^+(T\mathbb{R}^n) $
(one can take $c$ of the form $c(t, x) = (d(t, x), x) \in \mathsf{Sym} ^+(\mathbb{R}^n)
\times \mathbb{R}^n$). The lightlike metric is defined
by $g( \frac{\partial c} {\partial x^i}, \frac{\partial c} {\partial x^j}) =
c(t, x) ( \frac{\partial } {\partial x^i}, \frac{\partial } {\partial x^j})$ (this last
expression just means application of the scalar product $d(t, x)$
to $(\frac{\partial } {\partial x^i}, \frac{\partial } {\partial x^j})$).
If one takes $X = \frac{\partial } {\partial t}$ as a vector field tangent
to the null direction, then $L_X g ( \frac{\partial c} {\partial x^i}, \frac{%
\partial c} {\partial x^j})= \frac{\partial c} {\partial t} (\frac{\partial
} {\partial x^i}, \frac{\partial } {\partial x^j})$.  Now,  for a given
$x$,  $t \to c(t, x)$ is a parameterization of
$\mathcal C_x$ which is by our   definition of a regular GCS,
an embedded 1-dimensional manifold. Hence
  $\frac{\partial c%
} {\partial t}$  (seen as  element of $\mathsf{Sym} (\mathbb{R}^n)$)
does not vanish  which  shows that
the associated lightlike structure is always
 nowhere transversally
Riemannian. The lightlike metric $g$ is generic iff $L_Xg$ is non-degenerate on the space
generated by the
$\frac{\partial c} {\partial x^i}$'s. This is equivalent to that $\frac{\partial c%
} {\partial t}$ is non-degenerate, that is $\mathcal C$ is generic.


 \subsubsection{From   lightlike  to GCS structures}

We will  introduce a notion of simple  lightlike manifold ensuring that it comes from a GCS.

\begin{definition} A lightlike manifold $(\mathcal{V},
g)$ is said to be  \textbf{simple} if

a) There is a Hausdorff manifold
$M$ and a submersion $\pi: \mathcal{V} \to M$, such that the
connected components of its levels are the leaves of the null foliation
$\mathcal{N}$.

 b)   $\mathcal{C}$ is a regular GCS on $M$, where for $x \in M$,
 $\mathcal C_x$ is the set of all scalar products obtained from
  the projections $T_y\mathcal V \to T_xM$, where $y \in \mathcal V$
 is such that $x = \pi(y)$.

 \end{definition}

It is not so easy to formulate directly condition (b)  by means of $(\mathcal{V}%
, g)$ only (without refereeing to $M$), but the condition implies in particular that $(\mathcal{V}, g)$ is
nowhere transversally Riemannian.
Conversely, and this is the point, a nowhere transversally Riemannian
lightlike manifold is locally simple: any point admits a simple
neighborhood.

\medskip

Summarizing: there is a one to one correspondence between GCS structures and
simple lightlike ones, the generic in one hand correspond to the generic in
the other, and locally any nowhere transversally Riemannian lightlike metric
gives rise to a GCS.

\begin{example}
For the classical conformal sphere $\mathbb{S}^n$, the associated
lightlike manifold $\mathcal{V}$ is the  Minkowski lightcone
$$\mathsf{Co}^{n+1}= \{x = (x^1, \ldots, x^{n+2}) \in \R^{n+2} / q(x)= 0, x^{n+2}>0 \}$$
  seen as a lightlike submanifold in the Minkowski space $(\mathbb{R}^{n+2}, q)$,
 where $q(x)= (x^1)^2 + \ldots +(x^{n+1})^2 -  (x^{n+2})^2  $.

\end{example}

\subsubsection{Sub-rigidity}

A lightlike structure is an $H$-structure for $H$ the orthogonal group of
the standard lightlike scalar product $(x^1)^2 + \ldots + (x^{n-1})^2$ on $%
\mathbb{R}^n$. This structure has infinite type in Cartan's
terminology, equivalently it is not rigid in Gromov sense. We
discussed in \cite{BZ} subrigidity, a weaker property, that may be
satisfied by lightlike metrics. \textit{For $i <d$, a geometric
structure is $(d,i)$ \textbf{subrigid}, if any $d$-isometry which
has a trivial $i$-jet at some point has in fact a trivial
$(i+1)$-jet at that point}. In particular, $(d+1, d)$ subrigidity
coincides with usual $d$-rigidity.

\subsubsection{Isometry groups}

Let us call a \textit{transvection} of $(\mathcal{V}, g)$ any map $\mathcal{V%
} \to \mathcal{V}$ sending each leaf of $\mathcal{N}$ to itself. A
transvection is not necessarily isometric. In fact, any point admits in its
neighborhood a non-singular vector field generating (local) transvections,
iff $(\mathcal{V}, g)$ is transversally Riemannian.

If $(\mathcal{V}, g)$ is simple, then we have a group morphism $\mathsf{Iso}
(\mathcal{V}, g) \to \mathsf{Iso} (M, \mathcal{C})$. Its kernel is $\mathsf{%
Iso} ^{Tr}(\mathcal{V}, g)$, the group of isometric transvections. In the
simple case, $\mathsf{Iso} ^{Tr}(\mathcal{V}, g)$ does not contain one
parameter groups, but we can not conclude it is discrete, for instance
because one does not know if $\mathsf{Iso} (\mathcal{V}, g)$ is a Lie group.

Now, comparison between infinitesimal isometry groups of $(\mathcal{V}, g)$
and $(M, \mathcal{C})$ is even more complicated. We can however, as stated
in \cite{BZ}, relate subrigidity of $(\mathcal{V}, g)$ to the rigidity of $%
(M, \mathcal{C})$. Our second main result in the present article
will be to provide a proof of $(d+2, d) $ subrigidity of lightlike
metrics based on Liouville Theorem for GCS:

\begin{theorem}
\label{subrigid} In dimension $\geq 4$, a generic lightlike metric
is $(d+2, d) $ subrigid for $d \geq 1$, that is a (d+2)-isometry at
a given point with a trivial $d$-jet has a trivial $(d+1)$-jet. In
particular, an isometry with a trivial 1-jet at some point has a
trivial infinite jet.
\end{theorem}

The proof will be given in \S \ref{subrigidity_lightlike}. The general case
is no more difficult than that of $d=1$, that is $(3, 1)$-subrigidity.  We will start
 giving a detailed proof  in this last case and  show afterwards adaptations to the higher order  case
$d>1$.

\subsection{Remarks on other aspects}

Many other natural questions can be asked about both local and global
properties of GCS. For instance, one may try to weaken the genericity
condition in Theorems \ref{Liouville} and \ref{subrigid}, and also study
global properties of isometric actions preserving GCS from the point of view
of a global rigidity, say by asking a conjecture of Lichnerowicz type (see
\cite{gromov 1, gromov2, frances2}). We will here briefly discuss the
following other aspects:

\subsubsection{Pseudo-Riemannian case}

If one replaces $\mathsf{Sym} ^+$ by $\mathsf{Sym} ^*$, that is the space of
non-degenerate quadratic forms (i.e. scalar pseudo-products) then one gets
pseudo-Riemannian GCS that are defined similarly by giving a curve in each
$\mathsf{Sym} ^*(T_xM)$, for $x \in M$.  Theorem \ref{Liouville} seems to extend to
  this wider
framework. Indeed, all algebraic and local computations in Sections
\ref{Section.Braid} and \ref{Section.Liouville} apply in this situation, since they do
not assume positiveness but rather
non-degeneracy of metrics. However,  for the proof of  Theorem  \ref{Liouville},
positiveness is required  in particular to treat
the periodic case
\ref{subsection.periodic}.

\subsubsection{Anosov flows}

Let us give hints that Anosov flows always preserve GCS (of Riemannian
type), although they never preserve classical Riemannian conformal
structures (see for instance \cite{Has-Kat} for basic notions). Indeed, this
will be a particular case of the general construction of \ref{first}. The
point is that, one can choose the initial Riemannian metric $g_0$ so that
the corresponding family ${\phi^t_*}g_0$ defines a regular GCS.
Essentially, for any $x$, $t \to (\phi^t_*g_0) (x) \in \mathsf{Sym} ^+(T_xM)
$ is a properly embedded curve $\mathcal{C}_x$, and thus $\mathcal{C}=
\cup_x \mathcal{C}_x$ is a submanifold in $\mathsf{Sym} ^+(TM)$. To ensure
this, one has to start with an adapted $g_0$, that is, it is contracted on
the stable bundle, and expanded on the unstable one.

Regarding genericity, let us make the following technical assumption (which
it seems that one can overcome). Denote by $X$ the generating vector field
of $\phi^t$. Then assume that $\phi^t$ preserves a smooth supplementary
sub-bundle $E \subset TM$, i.e. $TM = \mathbb{R} X \oplus E$ (such an $E$
must be the sum of the stable and unstable bundles). Say $E$ is defined by a
1 differential form $\eta$. Assume $g_0(X, X) = 1$, and consider now the GCS
defined by $\phi^t_*g_0 + f(t) \eta \otimes \eta$, with $f(t)$ and $\frac{%
\partial f}{\partial t}$ positive for any $t$. This GCS is generic.

\subsubsection{A Geometric structure?}

\label{geometric_structure?}

 In general, GCS are neither $H$-structures
in Cartan sense nor geometric structures in the Gromov sense
(see \cite{gromov 1, gromov2, Ballmann, candel 2})! We already saw
that a GCS $\mathcal{C}$ is an $H$-structure iff it is
infinitesimally homogeneous: all the curves $\mathcal{C}_x \subset
\mathsf{Sym} ^+(T_xM)$ are linearly equivalent to a same curve $\mathcal{C}%
_0 \subset \mathsf{Sym} ^+(\mathbb{R}^n)$, when $x \in M$ (§\ref{second}).

Now, more generally, one may ask in which situations $\mathcal C$  can be   naturally seen
as a geometric structure in the
Gromov sense?  We will not investigate this question in the present article since
it hides many technical difficulties. Let us just say that roughly speaking,
and at a formal level, one considers
 $\mathcal X$,
the space of non-parameterized  curves $\R \to \mathsf{Sym}^+(\R^n)$ of a given regularity
 $C^k$, that
is the quotient space of $C^k(\R, \mathsf{Sym}^+(\R^n)$
by the $\mathsf{Diff}^k(\R)$ -right composition action.  Let $\mathcal X^*$ be the subspace of those curves
whose image is an embedded 1-submanifold in $\mathsf{Sym}^+(\R^n)$.  The group $\GL(n, \R)$ acts
on both $\mathcal X$ and $\mathcal X^*$.
Let us restrict ourselves to the case of structures  $\mathcal C \to M$
that are  trivial topological fibrations with fiber
$\R$ and where  $M$ is an open subset of $\R^n$.   Such a $\mathcal C$
 is equivalent to giving
a map $\sigma: M \to \mathcal X^*$.
Roughly, one may think of
$\mathcal C$ as  a geometric structure in the Gromov sense, if
 the image of
$\sigma$ is  contained in a $\GL(n, \R)$-invariant subset $\Sigma \subset
\mathcal X^*$, which is a finite dimensional manifold.
It is not clear how to formulate a general statement about a situation where such a $\Sigma$
exists. Let us however
notice the following  simple example. Consider $d$ an integer, and let
$\Sigma^\prime$  be the set of elements of $\mathcal X$ given by
polynomial maps $\R \to \mathsf{Sym}(\R^n)$ of degree $\leq d$, and
 take  $\Sigma = \Sigma^\prime \cap \mathcal X^*$.

\subsubsection{A-type?} \label{A-type?}
Observe now that
in order to get a geometric structure of algebraic type (A-type), as defined
in \cite{gromov 1} (see also \cite{gromov2, Ballmann}), one needs $\Sigma$
to be an algebraic manifold and the $\mathsf{GL}(n, \mathbb{R})$-action on
it algebraic (see \cite{Ballmann}).

But, rigid geometric structures of algebraic type satisfy the Gromov's open
dense orbit Theorem, that is if the isometry pseudo-group of the structure
has a dense orbit, then this one is open! In other words an open dense
subset is locally homogeneous (see \cite{gromov 1, gromov2, benoist, zeghib}%
). However, one can see in the previous Anosov case that there are examples
where such a local homogeneous subset can not exist.
 We then conclude that
there is no  way to see such a  GCS as a geometric structure of algebraic type!

\section{Some preliminaries}

\subsection{Case of $H$-structures}

\label{H.structure} Let $H \subset \mathsf{GL}(n, \mathbb{R})$ be a closed
subgroup and ${\mathfrak{h}} \subset \mathsf{End}(\mathbb{R}^n)$ its Lie
algebra. Recall that the space ${\mathfrak{h}}_d$ of $d$-prolongations is
that of symmetric $(d+1)$-multi-linear maps $A: \mathbb{R}^n \times \ldots
\times \mathbb{R}^n \to \mathbb{R}^n$, such that for any given $(u_1,
\ldots, u_d)$, the endomorphism $u \to A(u, u_1, \ldots, u_d)$ belongs to ${%
\mathfrak{h}}$.  If  for some $d \geq 1$, ${\mathfrak{h}}_d =0$, one says that $H$
has finite type, with order the smallest such $d$.

\subsubsection{Algebraic structure}

\begin{lemma}
Let $\mathcal{C}_0$ be a connected curve in $\mathsf{Sym} ^+(\mathbb{R}^n)$
and $H$ the connected component of its stabilizer in $\mathsf{GL}(n, \mathbb{%
R})$. Then $H$ is semi-direct product $P \ltimes K$, where $K$ is compact
and acts trivially on $\mathcal{C}_0$, and $P$ is either trivial or a one
parameter group acting transitively on $\mathcal{C}_0$.
\end{lemma}

\begin{proof}
$\mathcal{C}_0$ inherits from $\mathsf{Sym} ^+(\mathbb{R}^n)$ a Riemannian
metric (see \ref{metric.on.metrics}), and so by taking its parametrization by
arc length, it becomes isometric to an open interval of $\mathbb{R}$. In the
case where $\mathcal{C}_0$ is a proper interval, its length is finite and
hence it has limit endpoints, which are fixed by $H$, and thus $H$ is
compact in this case. Let us now consider the case where $\mathcal{C}_0$ is
isometric to $\mathbb{R}$.

We have a representation $\rho: H \to \mathsf{Iso} (\mathbb{R})$. The kernel
$K$ of $\rho$ is compact since it is a closed subgroup in the orthogonal
group $O(b)$, for any $b \in \mathcal{C}_0$.

Since $H$ is connected, $\rho(H)$ is either trivial or coincides with the
translation group of $\mathbb{R}$. It then follows that if $H$ is not
compact, then $H/K \sim \mathbb{R}$. In this case, let $P$ be any one
parameter group that projects onto $\mathbb{R}$ (to see it exists take the
one parameter group generated by any vector not in the Lie subalgebra of $K$%
). Thus $H$ is a semi-direct product $P \ltimes K$.
\end{proof}

\subsubsection{Finiteness of type}

Write $P = \exp t R$, and let $\langle, \rangle$ be a scalar product
preserved by $K$ (as in the lemma above). An element of the Lie algebra $%
\mathfrak{h}$ of $H$ has the form $C+ \alpha R$, where $C$ is antisymmetric (%
$C = -C^*$). A 2-prolongation $A: \mathbb{R}^n \times \mathbb{R}^n \times
\mathbb{R}^n \to \mathbb{R}^n$ of $\mathfrak{h}$ is symmetric and satisfies
that $W \to A(U, V, W)$ belongs to $\mathfrak{h}$ for any $U, V$. Therefore $%
A$ satisfies a relation $$\langle A( U,V,W) ,W^\prime \rangle + \langle A(
U,V,W^\prime) ,W \rangle =K( U,V) \langle (R + R^*) W, W^\prime \rangle $$
for some $K$. As it will be seen later on, this is exactly the equation (\ref%
{tresse}) in the generalized Braid Lemma \ref{GeneralizedBraid} with $J =
\langle, \rangle$ and $J^\prime (., .) = \langle (R+R^*).,. \rangle$. By
this lemma, it follows that $\mathfrak{h}$ has type $\leq$ 2 if  the form $%
J^\prime $ is non-degenerate.

\subsection{Case of periodic curves}
\label{subsection.periodic}

\subsubsection{Metric on $\mathsf{Sym} ^+$}

\label{metric.on.metrics} Let $E$ be a vector space of dimension $n$. Its
space of Euclidean Riemannian metrics (i.e. scalar products) $\mathsf{Sym}
^+(E)$ admits itself a canonical Riemannian (but no longer Euclidean)
metric. To see it, observe first that $\mathsf{Sym} ^+(E)$ is an open set in
$\mathsf{Sym} (E)$, and hence the tangent space $T_b (\mathsf{Sym} ^+(E))$
at any point $b$ can be identified with $\mathsf{Sym} (E)$.
But a scalar product $b$ defines a scalar product $\bar{b}$ on $\mathsf{Sym}
(E)$: if $(e_i)$ is a $b$-orthonormal basis, then $e_i^* \otimes e_j^*$ is a
$\bar{b}$-orthonormal basis, where $(e_i^*)$ is the dual basis (one has to
check this does not depend on the basis).
Now, endow $T_b \mathsf{Sym} ^+(E)$ with $\bar{b}$.
Clearly, if $F$ is another vector space, then any isomorphism $E \to F$
induces an isometry $\mathsf{Sym} ^+(E) \to \mathsf{Sym} ^+(F)$.
In fact, $\mathsf{Sym} ^+(E)$ is a symmetric space $\mathsf{GL}(E) / \mathsf{%
O}(b)$, where $\mathsf{O}(b)$ is the orthogonal group of any $b \in \mathsf{%
Sym} ^+(E)$.

As an example for $E = \mathbb{R}$, one gets the metric $\frac{dx^2}{x^2}$ on $\mathbb{R}%
^*$, and for $E= \mathbb{R}^2$, one gets the direct product $\mathbb{H}^2 \times \mathbb{%
R}$ (where $\mathbb{H}^2$ is the hyperbolic plane).

\subsubsection{Topology}

\begin{lemma}
Let $\mathcal{C}$ be a GCS on $M$ and assume that for some $x_0 \in M$, $%
\mathcal{C}_{x_0}$ is a circle, i.e. a connected compact 1-manifold. Then,
the same is true for nearby points. More precisely, there is a neighborhood
$V$ of $\mathcal{C}_{x_0}$ in $\mathcal{C}$ and $U$ a neighborhood of $x_0$
such that $\pi: V \to U$ is a Seifert fibration.
\end{lemma}

\begin{proof}
Let $I$ be a small arc in $M$ containing $x_0$, then $S = \pi^{-1}(I)$ is a
surface containing $\mathcal{C}_{x_0}$. Let $S_0$ be the connected component
of $\mathcal{C}_{x_0}$ in $S$. For $I$ small enough, $S_0$ is a tubular
neighborhood of $\mathcal{C}_{x_0}$ in $S$, and it is thus an annulus or a
Moebius strip around $\mathcal{C}_{x_0}$.

Let us start considering
  the annulus case. When $x$ runs over $I$, the connected components of the $\mathcal{C}_x$ in $S_0$ determine a
1-dimensional foliation $\mathcal{F}$ of $S_0$. But, each $\mathcal{C}_x$ is
closed in $\mathcal{C}$ and hence each $\mathcal{F}$-leaf is closed. But
such a foliation on the annulus is trivial, i.e. a trivial fibration on the
interval

Now, consider the same foliation $\mathcal{F}$, but on a neighborhood $V =
\pi^{-1}(U)$ in $\mathcal{C}$, where $U$ is a small neighborhood of $x_0$
in $M$. Since $U$ can be generated by arcs, leaves of $\mathcal{F}$ are all
closed. But the holonomy of $\mathcal{C}_{x_0}$ in $U$ is trivial, since it
is so above any interval. Hence the foliation is a fibration.

Consider now the case where for some arcs $I$, $\pi^{-1}(I)$ is a Moebius strip. Then, on such a   surface, the foliation
$\mathcal F$ is
  a Seifert fibration with monodromy $\mathbb{Z}/2 \mathbb{Z}$.   As above, generate a
  neighborhood $U$ by arcs such that the holonomy on each of them is either trivial
  or has order 2. It follows that the  (global) holonomy has order 2, and hence $\mathcal F$
  is given by a Seifert fibration.

\end{proof}

\subsubsection{Geometry}

\begin{proposition}
If $\mathcal{C}$ has a circle fiber $\mathcal{C}_{x_0}$, then it is 1-rigid
at $x_0$. In fact, $\mathcal{C}$ determines naturally a Riemannian metric
near $x_0$.
\end{proposition}

\begin{proof}
For all $x$, $\mathcal{C}_x$ is a circle in $\mathsf{Sym} ^+(T_xM)$. Consider an arc
length parameterization $t \in [0, l] \to f(t) \in \mathcal{C}_x$, where $l$
is the length of $\mathcal{C}_x$ ($f$ is defined up to a choice of an
origin). The mean $\int f(t)dt$ is a canonically defined element of $\mathsf{%
Sym} ^+(T_xM)$, call it $g_x$. Since $\pi$ is a smooth fibration, $g_x$
depends smoothly on $x$, that is $g$ is a smooth Riemannian metric defined
on a neighborhood of $x_0$.

One then verifies that a $d$-isometry of $\mathcal{C}$ is a $d$-isometry for
$g$. In order to check it, one considers the mapping which associates $g$ to
$\mathcal{C}$, say $F: \mathcal{G }\to \mathcal{M}$, defined on the space of
GCS with circle fibers, and having as a target the space of Riemannian
metrics $\mathcal{M }$. Locally, an element of $\mathcal{G}$ is a mapping $M
\to \mathcal{L}$, where $\mathcal{L}$ is the space of arc-length
parameterized circle maps $\mathbb{S}^1 \to \mathsf{Sym} (\mathbb{R}^n)$.
The mapping $F: \mathcal{G }\to \mathcal{M}$ is just a mean, and therefore
smooth. From all these constructions follows that if $\mathcal{C}$ and $%
\mathcal{C}^\prime$ have contact up to order $d$ at $p$, then the same is
true for $g = F(\mathcal{C})$ and $g^\prime = F(\mathcal{C}^\prime)$.
Applying this to $\mathcal{C}^\prime = \phi^*(\mathcal{C})$ yields that a $d$%
-isometry for $\mathcal{C}$ (at $p$) is a $d$-isometry for $g$.

Finally, by 1-rigidity of Riemannian metrics (that is a 2-isometry with
trivial 1-jet has a trivial 2-jet) we deduce that $\mathcal{C}$ is 1-rigid.
\end{proof}

\section{A generalized Braid Lemma}
\label{Section.Braid}

\label{Braid} The classical well known Braid Lemma (see for instance \cite%
{berger}) states:

\begin{lemma}
\label{Braidlemma}[Braid Lemma] If $L$ is a trilinear map $E\times E\times E
\to E$ on a vector space $E$, such that $L$ is symmetric on the two first
variables and skew-symmetric on the two last ones, then $L=0$. In
particular, if $A$ is a bilinear map $E\times E \to E$ such that
\begin{equation*}
<A(U,V),W>+<A(U,W),V>=0 \text { for all } U,V \text{ and }W \text{ in }E,
\end{equation*}
where $\langle, \rangle$ is a Euclidean scalar product, then $A=0$.

If fact this is also true for pseudo-scalar products, that is for $\langle,
\rangle$ replaced by any non-degenerate symmetric bilinear form.
\end{lemma}

This statement is equivalent to the vanishing of 1-prolongations of the
orthogonal group $O(E, \langle, \rangle)$, and thus to the 1-rigidity of a
Riemannian structures.

We are going here to give a generalized Braid Lemma adapted to GC
structures, which is in fact a slight generalization
of the classical result on vanishing  of second prolongations of
$\mathfrak{co}(n)$, see  for instance \cite{Ballmann}  and (\cite{sternberg}, page 335). Now, $A$ will be a trilinear symmetric map $E \times E \times E
\to E$, where $E $ is a vector space which will be always assumed to have
dimension $\geq 3$.

\begin{proposition}
\label{GeneralizedBraid} [Generalized Braid Lemma] Let $A$ be a symmetric
3-linear vectorial form $E \times E \times E \to E$ satisfying:
\begin{equation}
J( A( U,V,W) ,W^\prime) +J( A( U,V,W^\prime) ,W) =K( U,V) J^{\prime }( W,
W^\prime)  \label{tresse}
\end{equation}
where $J, J^\prime$ and $K$: $E\times E \to \mathbb{R}$ are some symmetric
bilinear forms.

If $J$ and $J^\prime$ are non-degenerate, then $A = 0$.
\end{proposition}

\begin{proof}
A direct computation gives us:
\begin{eqnarray}
K( U,V) J^{\prime }( W,W^{\prime }) &+& K( W,W^{\prime }) J^{\prime }( U,V)
\label{KJprime} \\
&=&K( U,W) J^{\prime }( V,W^{\prime }) +K( V,W^{\prime }) J^{\prime }( U,W)
\notag
\end{eqnarray}
(One just replaces each term as $K( U,V) J^{\prime }( W,W^{\prime })$ by its
equivalent in the right hand of (\ref{tresse}), and uses the fact that $A$
is symmetric).

Now let $W_1$ and $W_2$ be two $J^\prime$-orthogonal vectors: $J^\prime(W_1,
W_2) = 0$. Let $W_3$ be a third vector $J^\prime$-orthogonal to $\mathbb{R}
W_1 + \mathbb{R} W_2$ and $J^\prime(W_3, W_3) \neq 0$. Such $W_3$ exists
because $\dim E \geq 3$ and $J^\prime$ is non-degenerate. We have:
\begin{eqnarray*}
K( W_1,W_2) J^{\prime }( W_3,W_3 ) &+& K( W_3,W_3) J^{\prime }(W_1,W_2) \\
&=& K( W_1,W_3) J^{\prime }(W_2,W_3) +K( W_2,W_3) J^{\prime }(W_1,W_3)
\end{eqnarray*}
which implies $K(W_1, W_2) = 0$.

Write $K(U, V) = J^\prime(U, P(V))$, where $P$ is a $J^\prime$-symmetric
endomorphism of $E$.

Let $W_1$ with $J^\prime(W_1, W_1) \neq 0$, and denote by $W_1^\perp$ its $%
J^\prime$-orthogonal. It follows that $P(W_1) $ is orthogonal to $W_1^\perp$%
, and hence $P(W_1) \in \mathbb{R} W_1$, that is $W_1$ is an eigenvector of $%
P$.

Thus $P$ has all vectors $W_1$ with non-vanishing $J^\prime(W_1, W_1)$ as
eigenvectors. It follows that $P$ is a homothety, that is $K = \alpha
J^\prime$ for some $\alpha \in \mathbb{R}$.

Now, using (\ref{KJprime}) for $V=U$, $W^{\prime }=W$ and $J^{\prime }(U,W)=0
$, we get:
\begin{equation*}
\alpha J^{\prime }(U,U)J^{\prime }(W,W)=0
\end{equation*}%
which implies $\alpha =0$ (since we can easily choose $U$ and $W$ with
non-vanishing (square) $J^{\prime }$-norm). Therefore, (\ref{tresse})
becomes
\begin{equation*}
J(A(U,V,W),W^{\prime })+J(A(U,V,W^{\prime }),W)=0
\end{equation*}%
which implies by the classical Braid Lemma that $A=0$.
\end{proof}

\section{Proof of the generalized Liouville Theorem}
\label{Section.Liouville}

\subsection{Set-up of the problem}

Let $(M, \mathcal{C})$ be a GC manifold. The investigations in the present
section are local in nature, so the manifold $M$ can be identified with an
open set in $\mathbb{R}^n$ with coordinates $(x^1, \ldots, x^n)$. In fact,
we will work on a small neighborhood of a fixed point $p$.

So far, we have studied the situation where a component of $\mathcal{C}_p$
is a circle, and proved 1-rigidity in this case.

So we will now consider the opposite situation where all components of $%
\mathcal{C}_{p}$ are injective images of $\mathbb{R}$. We choose one
component and analyze $\mathcal{C}$ around it. The projection $\pi $ is not
necessarily a locally trivial fibration, but restricting to a small
neighborhood of $p$ (that we will still denote $M$), any neighbourhood of a
bounded arc of $\mathcal{C}_{p}$ can be parameterized by a map
\begin{equation*}
J:M\times I\rightarrow J(x,r)\in \mathsf{Sym}^{+}(T_{x}M)
\end{equation*}%
where $I$ is a bounded interval of $\mathbb{R}$. We can also assume that $%
r\rightarrow J(x,r)$ is an arc length parameterization, for any $x$,
although we do not need it. So locally,
\begin{equation*}
J(x,r)=\sum_{i,j}a_{ij}(x,r)dx^{i}dx^{j}
\end{equation*}%
We will always assume that
 the associated lightlike structure is
  nowhere transversally
Riemannian,  that  $\partial _{r}J\neq 0$.

\subsubsection{Isometries}

\label{isometries} For $\phi$ a diffeomorphism of $M$,  $\phi_x^\prime$ denotes
its derivative at $x$.

A diffeomorphism $\phi$ is isometric if its natural action on $\mathsf{Sym}
^+(TM)$ preserves $\mathcal{C}$, that is $\phi_x^\prime (\mathcal{C}_x) =
\mathcal{C}_{\phi(x)}$.
If the parameterization $J$ were global, then the isometric property implies
the existence of a re-parameterization $(x,r)\rightarrow k(x,r)\in \mathbb{R}
$, such that:
\begin{equation}
J(\phi (x),k(x,r))((\phi _{x}^{\prime })(U),(\phi _{x}^{\prime
})(V))-J{(x,r)}(U,V)=0,\forall \;U,V\;\hbox{vector fields}.  \label{isometry}
\end{equation}
Remark that, although we will not use it,  if the
$\mathcal{C}$-curves are parameterized by arc length, then $k$ has
the form $k(x, r) = \delta(x) + r$.

Now, if the parameterization is not global, one just has to take care of the
domains of definition; the same equation remains true. Actually, one has a
map $(x, r) \in M_1 \times I \to (\phi(x), k(x, r)) \in M_2 \times K$, where
$K$ is another interval, $M_1$ and $M_2$ are open subsets of $M$. However,
for the sake of simplicity of notation, we will argue as if the
parameterization is global, say $I = K$, and also $M_1 = M_2 =M$.

\subsection{Notation} \label{notations}

${}$

The notation $\phi _{x}^{\prime }$ designs the total derivative of the diffeomorphism $%
\phi $. Second and third total derivatives are denoted $\phi
_{x}^{\prime \prime }$ and $\phi _{x}^{\prime \prime \prime }$
respectively, the higher ones of order $m,m\in \mathbb{N}^{\ast }$,
are denoted $\phi _{x}^{(m)}$.
For a function $a$ on $(x,r)$ we denote the derivative with respect
to $x$
at a point $(p,r)$ by $\mathbf{D}_{(p,r)}a$ (i.e. the differential of $%
x\rightarrow a(x,r)$ where $r$ is fixed). We similarly denote the
same derivative of order $m$ by $\mathbf{D}_{_{(p,r)}}^{(m)}a$.
Regarding the derivative with respect to $r$ at a point $(p,r)$, we
just denote it $\partial _{(p,r)}a$.

\subsubsection{Infinitesimal isometries}

Assume now that $\phi(p) = p$. By definition, $\phi$ is a $d$-isometry at $p$
if $\mathcal{C}$ and its image $\phi^*(\mathcal{C})$ have a contact at order
$d$ along $\mathcal{C}_p$ (\ref{Rigidity}).
As in the classical case, one shows this is equivalent to the usual
vanishing condition up to order $d$, at $p$, of   the  equality (\ref{isometry}). More
precisely, for a given function $k$, and $U,V$ vector fields, let
\begin{equation*}
\Delta _{k}(U,V)(x,r)=J(\phi (x),k(x,r))((\phi _{x}^{\prime })(U),(\phi
_{x}^{\prime })(V))-J(x,r)(U,V)
\end{equation*}%
Then, $\phi $ is a $d$-isometry at $p$, if there exists a function $k$ such
that the derivatives with respect to $x$   up to order $(d-1)$ of $\Delta _{k}(U,V)$ vanish at $%
(p,r)$, for any vector fields $U$ and $V$ and any $r$.
Actually, it suffices to check this for $U$ and $V$ elements of a frame
field on $M$, for example the natural vector fields $\frac{\partial }{%
\partial x^i}$. In the sequel, we will take $U$ and $V$ to be combination
with constant coefficients of the $\frac{\partial }{\partial x^i}$.

\subsection{$(d+1)$-Isometries for $d\geq 2$}

If $\phi $ is an isometry of order $d+1$ at $p$, then for all
$r$
\begin{equation}
J(p,r)(U,V)=J(\phi \left( p\right) ,k\left( p,r\right) )(\phi
_{p}^{\prime }(U),\phi _{p}^{\prime }(V)),  \forall  \; U,V\in T_{p}M
\label{l=1}
\end{equation}%
Taking derivative (with respect to $x$) at $p$ gives, %
\begin{eqnarray}
\mathbf{D}_{(p,r)}J(W_{1})(U,V) &=&\mathbf{D}_{(\phi \left( p\right)
,k\left( p,r\right) )}J(\phi _{p}^{\prime }(W_{1}))(\phi _{p}^{\prime
}(U),\phi _{p}^{\prime }(V))  \label{l=2} \\
&&+\mathbf{D}_{(p,r)}k(W_{1})\mathbf{\partial }_{(\phi \left(
p\right) ,k\left( p,r\right) )}J(\phi _{p}^{\prime }(U),\phi
_{p}^{\prime }(V))
\notag \\
&&+J(\phi \left( p\right) ,k\left( p,r\right) )(\phi _{p}^{\prime \prime
}(U,W_{1}),\phi _{p}^{\prime }(V))  \notag \\
&&+J(\phi \left( p\right) ,k\left( p,r\right) )(\phi _{p}^{\prime
}(U),\phi _{p}^{\prime \prime }(V,W_{1}) \notag
\end{eqnarray}%
for all $U,V,W_{1}\in T_{p}M.$

\begin{lemma}
Let  $\phi $ be  an isometry of order $(d+1)$ at $p$ with a $d$-trivial jet
($jet_{p}^{d}(\phi )=1$) then, for all $r$
\begin{equation*}
k\left( p,r\right) =r\text{ and }\mathbf{D}_{(p,r)}^{(m)}k=0\text{ for }%
1\leq m\leq d-1
\end{equation*}
\end{lemma}

\begin{proof} This will follow from taking derivatives of  (\ref{l=1}) up to order $d-1$. But,
since $jet_{p}^{d}(\phi )=1$, one can argue as if $\phi$ was the identity, that is
  (\ref{l=1}) becomes%
\begin{equation} \label{l=1'}
J{(p,r)}(U,V)=J{(p,k\left( p,r\right) )}(U,V),\forall U,V\in T_{p}M
\end{equation}%
This formula   involves   $k$ only.

\bigskip

  This equality itself implies that $k(p,r)=r$. Indeed,
by     our hypotheses (in the beginning of the present  \S),
$\mathcal{C}_{p}$ is a 1-dimensional submanifold without compact components,
in particular, $r\rightarrow J(p,r)$ is injective, and hence $k(p,r)=r$.

\bigskip

To prove vanishing of ${\mathbf D}_{(p,r)}k$ (the differential of $k$
with respect to $x$), just differentiate   the   formula  (\ref{l=1'}) (for instance by  replacing
in   (\ref{l=2}))  and get
\begin{equation*}
\mathbf{D}_{(p,r)}J(W_{1})(U,V)=\mathbf{D}_{(p,r)}J(W_{1})(U,V)+\mathbf{D}%
_{(p,r)}k(W_{1})\mathbf{\partial }_{(p,r)}J(U,V)
\end{equation*}%
which means%
\begin{equation*}
\mathbf{D}_{(p,r)}k(W_{1})\mathbf{\partial }_{(p,r)}J(U,V)=0
\end{equation*}%
(Remember the notation ${\partial }_{(p,r)}$ introduced in \S \ref{notations}).
But, by definition of GCS, the curve $r\rightarrow J(p,r)\in \mathsf{Sym}%
^{+}(T_{p}M)$ is non-singular, and hence $(U,V)\rightarrow \mathbf{\partial }%
_{(p,r)}J(U,V)$ is a non-vanishing bilinear form, and so $\mathbf{D}%
_{(p,r)}k=0$.

\bigskip

Finally, vanishing of higher order derivatives is done by induction.
Assume  $\mathbf{D}_{(p,r)}^{(m)}k=0$ for
$m\leq l$ (with $ 1\leq l\leq d-2$), and
 take the derivative of order $l+1$
  of  the equality  (\ref{l=1'}) at $p$.
  All terms containing $%
\mathbf{D}_{(p,r)}^{(m)}k$ for $1\leq m\leq l$  disappear and
remains the equality
\begin{eqnarray*}
\mathbf{D}_{(p,r)}^{(l+1)}J(W_{1},...,W_{l+1})(U,V) &=&\mathbf{D}%
_{(p,r)}^{(l+1)}J(W_{1},...,W_{l+1})(U,V) \\
&&+\mathbf{D}_{(p,r)}^{(l+1)}k(W_{1},...,W_{l+1})\mathbf{\partial }%
_{(p,r)}J(U,V)
\end{eqnarray*}%
for all $W_{1},...,W_{l+1}$ in $T_{p}M$, which implies  (as in the case of
$\mathbf{D}_{(p,r)}^{}k$)
\begin{equation*}
\mathbf{D}_{(p,r)}^{(l+1)}k=0
\end{equation*}
\end{proof}

\begin{lemma}
Let $\phi $ be  an isometry of order $(d+1)$ at $p$ such that $jet_{p}^{d}(\phi )=1$. Then,
 $\phi _{p}^{(d+1)} $ satisfies
\begin{gather}
J(p,r)(\phi _{p}^{(d+1)}(U,W_{1},...,W_{d}),V)+J(p,r)(U,\phi
_{p}^{(d+1)}(V,W_{1},...,W_{d}))  \label{l=l} \\
=-\mathbf{D}_{(p,r)}^{(d)}k(W_{1},...,W_{d})\mathbf{\partial }_{(p,r)}J(U,V)
\notag
\end{gather}%
for any $r$ and all $W_{1},...,W_{d},U,V$ in $T_{p}M$.
\end{lemma}

\begin{proof}

Computation of the derivative at order $(d-1)$ of  (\ref{l=2})  at $p$, will be
drastically simplified by the fact that
$\phi_p ^{\prime }=Id,\phi _{p}^{(m)}=0$
   for $2\leq m\leq d$,
  $k(p, r)= r$ and
$\mathbf{D}_{(p,r)}^{(m)}k=0$ for $1\leq m\leq d-1$ (by the previous lemma),
and reduces  exactly to
\begin{gather*}
\mathbf{D}_{(p,r)}^{(d)}J(W_{1},...,W_{d})(U,V)=\mathbf{D}%
_{(p,r)}^{(d)}J(W_{1},...,W_{d})(U,V)+\mathbf{D}%
_{(p,r)}^{(d)}k(W_{1},...,W_{d})\mathbf{\partial }_{(p,r)}J(U,V) \\
+J(p,r)(\phi _{p}^{(d+1)}(U,W_{1},...,W_{d}),V)+J(p,r)(U,\phi
_{p}^{(d+1)}(V,W_{1},...,W_{d}))
\end{gather*}%
for all $W_{1},...,W_{d},U,V$ in $T_{p}M$, witch is exactly
(\ref{l=l}).
\end{proof}

\subsection{End of the proof of Theorem \protect\ref{Liouville}} With notations
of the previous lemma, we have to prove that $\phi _{p}^{(d+1)}=0$. This will indeed
follow form a straightforward application of the  Generalized Braid Lemma \ref{GeneralizedBraid}
to (\ref{l=l}).

For the sake of clarity, let us   first start with  the case $d=2$.
So  $\phi $ is a 3-isometry at $p$ with a trivial 2-jet, thus by
(\ref{l=l})
\begin{gather*}
J(p,r)(\phi _{p}^{\prime \prime \prime
}(U,W_{1},W_{2}),V)+J(p,r)(U,\phi _{p}^{\prime \prime
\prime }(V,W_{1},W_{2}) \\
=-\mathbf{D}_{(p,r)}^{(2)}k(W_{1},W_{2})\mathbf{\partial }%
_{(p,r)}J(U,V)
\end{gather*}%
Apply the Generalized Braid Lemma   with $A=\phi _{p}^{\prime \prime
\prime }$, $J=J(p,r)$, $K=\mathbf{D}_{(p,r)}^{2}k$ and $J^{\prime
}=-\partial _{(p,r)}J$, which is actually  non-degenerate by the genericity hypothesis
on $\mathcal{C}$. Then
conclude that $A= \phi _{p}^{\prime \prime \prime }=0$.

In the general case,  $d>2$, apply  the Generalized Braid Lemma  with  $J=J(p,r)$, $J^{\prime }=-\partial _{(p,r)}J$ and $K=\mathbf{D}%
_{(p,r)}^{(d)}k(.,.,W_{3},...,W_{d})$,   where   $W_{3},...,W_{d}$  are fixed  vectors in $%
T_{p}M$. One gets that  $A=\phi
_{p}^{(d+1)}(.,.,.,W_{3},...,W_{d})=0$. But since $W_3, \ldots W_d$
are arbitrary,
  $\phi _{p}^{(d+1)}=0$.

\section{Sub-rigidity of lightlike metrics, Proof of Theorem \protect\ref%
{subrigid}}
\label{subrigidity_lightlike}

\subsection{Setting of the problem}

Let $(\mathcal{V},g)$ be a lightlike $n$-dimensional manifold. Since
we are dealing with questions local in nature, so we can assume
$\mathcal{V}$ is a small chart domain, say $\mathcal{V} = M \times
I$ where $I$ is an interval. The factor $I$ corresponds to the
characteristic foliation tangent to the kernel of $g$. In an adapted
coordinate system $(x,t)=(x^1,x^2,...,x^{n-1},t)$ ($t$ corresponds
to $I$), the lightlike metric takes the form
\begin{equation*}
g_{(x,t)}=\sum_{i,j}a_{ij}(x,t)dx^idx^j
\end{equation*}
This gives for any fixed $r$, a Riemannian metric on $M \times \{ r \}$. By
endowing $T_xM$ with the scalar products $g_{(x, r)}, r \in I$, we get a GCS
on $M$, once we assume $g$ nowhere transversally Riemannian, that is $\frac{%
\partial}{\partial t} g_{(x, t)} \neq 0$ (see \ref{transversalyRiemaninnian}%
). Recall that $g$ is said to be generic if $\frac{\partial}{\partial t}
g_{(x, t)} = \sum_{i,j} \frac{\partial a_{ij}}{\partial t}(x,t)dx^idx^j$ is
non-degenerate.

A diffeomorphism $\Psi$ of $M$ has the form $\Psi=(\phi,\delta)$ where $%
\phi: M\times I\rightarrow M$ and $\delta: M\times I\rightarrow I$.

If $\Psi$ is isometric, then it preserves the $I$ foliation, and hence $\phi$
does not depend of $t$. Furthermore, for any $U$ and $V$ in $T_{(x, t)}%
\mathcal{V}$
\begin{equation}
g_{(x,t)}(U,V)=g_{\Psi(x,t)}(\Psi^\prime_{(x,t)}(U),\Psi^\prime_{(x,t)}(V))
\label{lightlikeisometry}
\end{equation}

A tangent vector $U \in T_{(x, t)}\mathcal{V}$ will be denoted $(U_M, U_I)
\in T_x M \times T_t I$.

\bigskip

As said after the  statement of the Theorem, we will start giving
the proof   in the case $d=1$ which consists in three steps. The
higher order case will be treated at \S \ref{higher_degree}.

\subsection{Step 1: a partial 1-rigidity}

\textit{If $\Psi$ is isometric up to order 2, with a trivial 1-jet at a
point $(p,r)\in \mathcal V$ then $\phi^{\prime\prime}_{(p,r)}=0$.}

\bigskip \noindent

\begin{proof}
If $\Psi=(\phi,\delta)$ is isometric up to order 2 then the equality of (\ref%
{lightlikeisometry}) holds for the derivatives at $(x,t)=(p,r)$. We have
\begin{eqnarray}
(g_{(x,t)}(U,V))^{\prime}_{(p,r)}(W)&=&\sum_{i,j}(a_{ij})^{%
\prime}_{(p,r)}(W)(U)_i(V)_j  \label{firstsidederivation}
\end{eqnarray}

In the other hand, if we denote a generic point $(x,t)$ by $v$
\begin{eqnarray*}
g_{\Psi(v)}(\Psi^\prime_{v}(U),\Psi^\prime_{v}(V))
=\sum_{i,j}a_{ij}(\Psi(v))(\Psi^\prime_{v}(U))_i(\Psi^\prime_{v}(V))_j
\end{eqnarray*}
a derivation gives
\begin{eqnarray}
&&\sum_{i,j}(a_{ij})^{\prime}_{\Psi(v)}(\Psi^\prime_{v}(W))(\Psi^%
\prime_{v}(U))_i(\Psi^\prime_{v}(V))_j
+\sum_{i,j}a_{ij}(\Psi(v))(\Psi^{\prime\prime}_{v}(U,W))_i(\Psi^%
\prime_{v}(V))_j  \label{secondsidederivation} \\
&&\text{\hspace{1.5 cm}}\space+\sum_{i,j}a_{ij}(\Psi(v))(\Psi^%
\prime_{v}(U))_i(\Psi^{\prime\prime}_{v}(V,W))_j  \notag
\end{eqnarray}
We have
\begin{equation*}
(\Psi^\prime_{v}(U))_i=(\phi^\prime_{v}(U))_i \text{ and }%
(\Psi^{\prime\prime}_{v}(U))_i=(\phi^{\prime\prime}_{v}(U))_i
\end{equation*}
using the triviality of the 1-jet of $\Psi$ we get
\begin{equation*}
\Psi(p,r)=(p,r),\Psi^\prime_{(p,r)}=Id
\end{equation*}
then (\ref{secondsidederivation}) becomes
\begin{eqnarray*}
&&\sum_{i,j}(a_{ij})^{\prime}_{(p,r)}(W)(U)_i(V)_j
+\sum_{i,j}a_{ij}(p,r)(\phi^{\prime\prime}_{(p,r)}(U,W))_i(V)_j
+\sum_{i,j}a_{ij}(p,r)(U)_i(\phi^{\prime\prime}_{(p,r)}(V))_j
\end{eqnarray*}

Therefore, the equality with (\ref{firstsidederivation}) gives
\begin{equation*}
\sum_{i,j}a_{ij}(p,r)(\phi^{\prime\prime}_{(p,r)}(U,W))_i(V)_j
+\sum_{i,j}a_{ij}(p,r)(U)_i(\phi^{\prime\prime}_{(p,r)}(V,W))_j=0
\end{equation*}
that is
\begin{equation*}
g_{(p,r)}(\phi^{\prime\prime}_{(p,r)}(U,W),V)
+g_{(p,r)}(U,\phi^{\prime\prime}_{(p,r)}(V,W))=0
\end{equation*}
By the Braid Lemma \ref{Braidlemma} we conclude that $\phi^{\prime\prime}_{(p,r)}=0$.
\end{proof}

\subsection{Step 2: the $\protect\phi$-part}

\textit{Assume the lightlike structure generic. If $\Psi=(\phi,\delta)$ is a
3-isometry at $(p,r)$ with a trivial 1-jet, then $\phi^{\prime\prime%
\prime}_{(p,r)}=0$. }

\bigskip \noindent

\begin{proof}
If $\Psi$ was a true isometry, then it acts, via $\phi$, on $M$ seen as the
quotient space of the characteristic foliation (in particular it does not
depend on $r$), and it preserves the GCS on it. The genericity hypothesis
allows one to apply Theorem \ref{Liouville} to conclude that $\phi^{\prime
\prime \prime}_{(p,r)} = 0$.

Now, we want to apply the same argument when $\Psi$ is merely isometric up
to order 3 at $(p, r) $ (and has a trivial 1-jet). The idea then is to show
that the diffeomorphism $x \to \varphi(x) = \phi(x, r) $ is a 3-isometry of
the GCS of $M$. The expected $k$-shift of $\varphi$ (see \ref{isometries})
will be nothing but $\delta$. In other words, we want $\varphi$ to satisfy
the following equation up to order 3 at $p$:
\begin{equation}  \label{equation.varphi}
g_{(x,r)}(U,V)=g_{(\varphi(x), \delta(x,
r))}(\varphi^\prime_{x}(U),\varphi^\prime_{x}(V))
\end{equation}

This property of $\varphi$, follows from the similar one of $\Phi$, that is,
it satisfies (\ref{lightlikeisometry}) up to order 3, and remembering that $%
\phi^{\prime \prime}_{(p, r)} = 0$ by the previous step. Indeed, let us
derive twice the equation satisfied by $\Psi$ at $\nu=(p,r)$:

\begin{equation*}
g_{(x,t)}(U,V)=g_{\Psi (x,t)}(\Psi _{(x,t)}^{\prime }(U),\Psi
_{(x,t)}^{\prime }(V))
\end{equation*}%
we get, for all $W_1, W_2 \in T_{(x, t} \mathcal V$,
\begin{eqnarray}
&&g_{\nu }^{W_{1},W_{2}}(U,V)=g_{\Psi (\nu )}^{\Psi _{\nu }^{\prime \prime
}(W_{1},W_{2})}(\Psi _{\nu }^{\prime }(U),\Psi _{\nu }^{\prime }(V))+g_{\Psi
(\nu )}^{\Psi _{\nu }^{\prime }(W_{1}),\Psi _{\nu }^{\prime }(W_{2})}(\Psi
_{\nu }^{\prime }(U),\Psi _{\nu }^{\prime }(V))  \label{3lightlikeisometry}
\\
&&\hspace{2cm}+g_{\Psi (\nu )}^{\Psi _{\nu }^{\prime }(W_{1})}(\Psi _{\nu
}^{\prime \prime }(U,W_{2}),\Psi _{\nu }^{\prime }(V))+g_{\Psi (\nu )}^{\Psi
_{\nu }^{\prime }(W_{1})}(\Psi _{\nu }^{\prime }(U),\Psi _{\nu }^{\prime
\prime }(V,W_{2}))  \notag \\
&&\hspace{2cm}+g_{\Psi (\nu )}^{\Psi _{\nu }^{\prime }(W_{2})}(\Psi _{\nu
}^{\prime \prime }(U,W_{1}),\Psi _{\nu }^{\prime }(V))+g_{\Psi (\nu )}^{\Psi
_{\nu }^{\prime }(W_{2})}(\Psi _{\nu }^{\prime }(U),\Psi _{\nu }^{\prime
\prime }(V,W_{1}))  \notag \\
&&\hspace{2cm}+g_{\Psi (\nu )}(\Psi _{\nu }^{\prime \prime }(U,W_{1}),\Psi
_{\nu }^{\prime \prime }(V,W_{2}))+g_{\Psi (\nu )}(\Psi _{\nu }^{\prime
\prime }(U,W_{2}),\Psi _{\nu }^{\prime \prime }(V,W_{1}))  \notag \\
&&\hspace{2cm}+g_{\Psi (\nu )}(\Psi _{\nu }^{\prime \prime \prime
}(U,W_{1},W_{2}),\Psi _{\nu }^{\prime }(V))+g_{\Psi (\nu )}(\Psi _{\nu
}^{\prime }(U),\Psi _{\nu }^{\prime \prime \prime }(V,W_{1},W_{2}))  \notag
\end{eqnarray}%
where
\begin{equation*}
g_{v}^{W}=\sum_{i,j}(a_{ij})_{v}^{\prime }(W)dx^{i}dx^{j}
\end{equation*}%
and
\begin{equation*}
g_{v}^{W_{1},W_{2}}=\sum_{i,j}(a_{ij})_{v}^{\prime \prime
}(W_{1},W_{2})dx^{i}dx^{j}
\end{equation*}%
But
\begin{equation*}
g_{v}(U,V)=g_{v}(U,V_{M})=g_{v}(U_{M},V)=g_{v}(U_{M},V_{M}),
\end{equation*}%
and the same thing for $g_{v}^{W}$ and $g_{v}^{W_{1},W_{2}}$, then (\ref%
{3lightlikeisometry}) becomes
\begin{eqnarray}
&&g_{\nu }^{W_{1},W_{2}}(U,V)=g_{\Psi (\nu )}^{\Psi _{\nu }^{\prime \prime
}(W_{1},W_{2})}(\Psi _{\nu }^{\prime }(U),\Psi _{\nu }^{\prime }(V))+g_{\Psi
(\nu )}^{\Psi _{\nu }^{\prime }(W_{1}),\Psi _{\nu }^{\prime }(W_{2})}(\Psi
_{\nu }^{\prime }(U),\Psi _{\nu }^{\prime }(V))
\label{3lightlikeisometryvarphi} \\
&&\hspace{2.1cm}+g_{\Psi (\nu )}^{\Psi _{\nu }^{\prime }(W_{1})}(\phi _{\nu
}^{\prime \prime }(U,W_{2}),\Psi _{\nu }^{\prime }(V))+g_{\Psi (\nu )}^{\Psi
_{\nu }^{\prime }(W_{1})}(\Psi _{\nu }^{\prime }(U),\phi _{\nu }^{\prime
\prime }(V,W_{2}))  \notag \\
&&\hspace{2.1cm}+g_{\Psi (\nu )}^{\Psi _{\nu }^{\prime }(W_{2})}(\phi _{\nu
}^{\prime \prime }(U,W_{1}),\Psi _{\nu }^{\prime }(V))+g_{\Psi (\nu )}^{\Psi
_{\nu }^{\prime }(W_{2})}(\Psi _{\nu }^{\prime }(U),\phi _{\nu }^{\prime
\prime }(V,W_{1}))  \notag \\
&&\hspace{2.1cm}+g_{\Psi (\nu )}(\phi _{\nu }^{\prime \prime }(U,W_{1}),\phi
_{\nu }^{\prime \prime }(V,W_{2}))+g_{\Psi (\nu )}(\phi _{\nu }^{\prime
\prime }(U,W_{2}),\phi _{\nu }^{\prime \prime }(V,W_{1}))  \notag \\
&&\hspace{2.1cm}+g_{\Psi (\nu )}(\phi _{\nu }^{\prime \prime \prime
}(U,W_{1},W_{2}),\Psi _{\nu }^{\prime }(V))+g_{\Psi (\nu )}(\Psi _{\nu
}^{\prime }(U),\phi _{\nu }^{\prime \prime \prime }(V,W_{1},W_{2}))  \notag
\end{eqnarray}

Since $\Psi $ has a trivial 1-jet at $\nu $, $\Psi (\nu )=(p,r)$, $\Psi
_{\nu }^{\prime }=Id$ and $\phi _{\nu }^{\prime \prime }=0$ (Step 1), and so
\begin{eqnarray}
g_{(p,r)}(\phi _{(p,r)}^{\prime \prime \prime }(U,W_{1},W_{2}),V)
&+&g_{(p,r)}(U,\phi _{(p,r)}^{\prime \prime \prime }(V,W_{1},W_{2}))
\label{generalizedBLfor lightlike0} \\
&=&-g_{(p,r)}^{\Psi _{(p,r)}^{\prime \prime }(W_{1},W_{2})}(U,V)  \notag
\end{eqnarray}%
But
\begin{eqnarray*}
g_{(p,r)}^{\Psi _{(p,r)}^{\prime \prime }(W_{1},W_{2})}(U,V)
&=&\sum_{i,j}(a_{ij})_{(p,r)}^{\prime }(\Psi _{(p,r)}^{\prime \prime
}(W_{1},W_{2}))(U)_{i}(V)_{j} \\
&=&\sum_{i,j}((\mathbf{D}a_{ij})_{(p,r)}(\phi _{(p,r)}^{\prime \prime
}(W_{1},W_{2}))+\delta _{(p,r)}^{\prime \prime
}(W_{1},W_{2})\partial
_{(p,r)}a_{ij})(U)_{i}(V)_{j} \\
&=&\sum_{i,j}\delta _{(p,r)}^{\prime \prime }(W_{1},W_{2})\partial
_{(p,r)}a_{ij}(U)_{i}(V)_{j} \\
&=&\delta _{(p,r)}^{\prime \prime
}(W_{1},W_{2})\partial_{(p,r)}g(U,V)
\end{eqnarray*}%
where%
\begin{equation*}
\partial_{(p,r)}g(U,V)=\sum_{i,j}\partial _{(p,r)}a_{ij}(U)_{i}(V)_{j}
\end{equation*}%
Thus (\ref{generalizedBLfor lightlike0}) gives
\begin{eqnarray}
g_{(p,r)}(\phi _{(p,r)}^{\prime \prime \prime }(U,W_{1},W_{2}),V)
&+&g_{(p,r)}(U,\phi _{(p,r)}^{\prime \prime \prime }(V,W_{1},W_{2}))
\label{generalizedBLforlightlike} \\
&=&-\delta _{(p,r)}^{\prime \prime
}(W_{1},W_{2})\partial_{(p,r)}g(U,V)  \notag
\end{eqnarray}%
Finally, apply the Generalized Braid Lemma to $J=g_{(p,r)},K=\delta
_{(p,r)}^{\prime \prime },J^{\prime }=-\partial_{(p,r)}g$ and
$A=\phi _{(p,r)}^{\prime \prime \prime }$ to get $\phi
_{(p,r)}^{\prime \prime \prime }=0$.
\end{proof}

\subsection{Step 3: the $\protect\delta$-part, end of proof of Theorem
\protect\ref{subrigid}}

\textit{Let $\mathcal{V}$ be a generic lightlike manifold. If
$\Psi=(\phi,\delta)$ is a 3-isometry at $(p,r)$ with a trivial 1-jet
at $(p,r)$ then $\Psi$ has a trivial 2-jet at $(p,r)$. }

\bigskip \noindent

\begin{proof}
By step 2 we have $\phi^{\prime\prime\prime}_{(p,r)}=0$, so (\ref%
{generalizedBLforlightlike}) gives
\begin{equation*}
\delta^{\prime\prime}_{(p,r)}
(W_{1},W_{2})\sum_{i,j}\partial_{(p,r)}a_{ij}(U)_i(V)_j=0
\end{equation*}
Remember that $g$ is nowhere transversally Riemannian, thus $%
\delta^{\prime\prime}_{(p,r)}=0$, and hence $\Psi$ has a trivial 2-jet.

\bigskip
Thus, Theorem \ref{subrigid} in the case $d=1$, that is
(3,1)-subrigidity of generic lightlike metrics, is fully proved.

\subsection{Proof in the case $d>1$}
\label{higher_degree}

  If
$\Psi =(\phi ,\delta )$ is isometric up to order $d+2$ at a point
$\left( p,r\right) \in \mathcal V$, then the equality of
(\ref{lightlikeisometry}) holds for the derivatives of order $d+1$
at $(x,t)=(p,r)$. A derivation of order $d$ of the left side gives
\begin{equation}
(g_{(x,t)}(U,V))_{(p,r)}^{(d)}(W_{1},W_{2},...,W_{d})=%
\sum_{i,j}(a_{ij})_{(p,r)}^{(d)}(W_{1},W_{2},...,W_{d})(U)_i(V)_j
\label{d derivation}
\end{equation}
for all $(W_{1},W_{2},...,W_{d})\in T_{(p,r)}\mathcal V$.  On  the
other hand, taking derivation of  the right side of
(\ref{lightlikeisometry}) at $(p,r)$, and using the fact that $\Psi
$ has a trivial $d$-jet at $(p,r)$, it remains
\begin{eqnarray*}
&&\sum_{i,j}(a_{ij})_{(p,r)}^{(d)}(W_{1},W_{2},...,W_{d})(U)_{i}(V)_{j}+%
\sum_{i,j}a_{ij}(p,r)(\phi
_{(p,r)}^{(d+1)}(U,W_{1},W_{2},...,W_{d}))_{i}(V)_{j}  \\&&+%
\sum_{i,j}a_{ij}(p,r)(U)_{i}(\phi
_{(p,r)}^{(d+1)}(V, W_{1},W_{2},...,W_{d}))_{j}
\end{eqnarray*}

Comparing with   (\ref{d derivation}), we get,
\begin{equation*}
\sum_{i,j}a_{ij}(p,r)(\phi
_{(p,r)}^{(d+1)}(U,W_{1},W_{2},...,W_{d}))_{i}(V)_{j}+%
\sum_{i,j}a_{ij}(p,r)(U)_{i}(\phi
_{(p,r)}^{(d+1)}(V, W_{1},W_{2},...,W_{d}))_{j}=0
\end{equation*}%
that is,
\begin{equation*}
g_{(p,r)}(\phi _{(p,r)}^{(d+1)}(U,W_{1},W_{2},...,W_{d}),V)+g_{(p,r)}(U,\phi
_{(p,r)}^{(d+1)}(V, W_{1},W_{2},...,W_{d}))=0
\end{equation*}%
By the Braid Lemma  \ref{Braidlemma} we conclude that $\phi _{(p,r)}^{(d+1)}=0$.

Now, if we derive (\ref{lightlikeisometry}) $(d+1)$-times at $(p,r)$, we get
for the left side%
\begin{equation}
\sum_{i,j}(a_{ij})_{(p,r)}^{(d+1)}(W_{1},W_{2},...,W_{d+1})(U)_{i}(V)_{j}
\label{d+1 derivation}
\end{equation}%
and for the right one%
\begin{eqnarray*}
&&\sum_{i,j}(a_{ij})_{(p,r)}^{(d+1)}(W_{1},W_{2},...,W_{d+1})(U)_{i}(V)_{j}+%
\sum_{i,j}(a_{ij})_{(p,r)}^{\prime }(\Psi
_{(p,r)}^{(d+1)}(W_{1},W_{2},...,W_{d+1}))(U)_{i}(V)_{j} \\
&&+\sum_{i,j}a_{ij}(p,r)(\Psi
_{(p,r)}^{(d+2)}(U,W_{1},...,W_{d+1}))_{i}(V)_{j}
\\
&&+\sum_{i,j}a_{ij}(p,r)(U)_{i}(\Psi
_{(p,r)}^{(d+2)}(V,W_{1},...,W_{d+1}))_{j}
\end{eqnarray*}%
for any $W_{1},W_{2},...,W_{d+1}\in T_{(p,r)} \mathcal V$, since
$\Psi $ has a trivial $d$-jet at $(p,r)$ and ${\phi
_{(p,r)}^{(d+1)}=0.}$  Writing equality between the two sides gives,
\begin{eqnarray*}
&&\sum_{i,j}a_{ij}(p,r)(\phi
_{(p,r)}^{(d+2)}(U,W_{1},...,W_{d+1}))_{i}(V)_{j}+\sum_{i,j}a_{ij}(p,r)(U)_{i}(%
\phi _{(p,r)}^{(d+2)}(V,W_{1},...,W_{d+1}))_{j} \\
&&=-\sum_{i,j}(a_{ij})_{(p,r)}^{\prime }(\Psi
_{(p,r)}^{(d+1)}(W_{1},W_{2},...,W_{d+1}))(U)_{i}(V)_{j}
\end{eqnarray*}%
that is
\begin{eqnarray}
g_{(p,r)}(\phi _{(p,r)}^{(d+2)}(U,W_{1},...,W_{d+1}),V) &+&g_{(p,r)}(U,\phi
_{(p,r)}^{(d+2)}(V,W_{1},...,W_{d+1}))  \label{generalizedBLfor lightlike d}
\\
&=&-g_{(p,r)}^{\Psi _{(p,r)}^{(d+1)}(W_{1},...,W_{d+1})}(U,V)  \notag
\end{eqnarray}%
But
\begin{eqnarray*}
g_{(p,r)}^{\Psi _{(p,r)}^{(d+1)}(W_{1},...,W_{d+1})}(U,V)
&=&\sum_{i,j}(a_{ij})_{(p,r)}^{\prime }(\Psi
_{(p,r)}^{(d+1)}(W_{1},...,W_{d+1}))(U)_{i}(V)_{j} \\
&=&\sum_{i,j}(\mathbf{D}a_{ij})_{(p,r)}(\phi
_{(p,r)}^{(d+1)}(W_{1},...,W_{d+1})) (U)_{i}(V)_{j} \\
&&+\sum_{i,j}\delta
_{(p,r)}^{(d+1)}(W_{1},...,W_{d+1})\partial _{(p,r)}a_{ij}(U)_{i}(V)_{j} \\
&=&\sum_{i,j}\delta _{(p,r)}^{(d+1)}(W_{1},...,W_{d+1})\partial
_{(p,r)}a_{ij}(U)_{i}(V)_{j} \\
&=&\delta _{(p,r)}^{(d+1)}(W_{1},...,W_{d+1})\partial _{(p,r)}g(U,V)
\end{eqnarray*}%
Thus (\ref{generalizedBLfor lightlike d}) becomes
\begin{eqnarray}
&&g_{(p,r)}(\phi _{(p,r)}^{(d+2)}(U,W_{1},...,W_{d+1}),V)+g_{(p,r)}(U,\phi
_{(p,r)}^{(d+2)}(V,W_{1},...,W_{d+1}))  \label{generalizedBLforlightlike} \\
&=&-\delta _{(p,r)}^{(d+1)}(W_{1},...,W_{d+1})\partial _{(p,r)}g(U,V)  \notag
\end{eqnarray}%
Applying the Generalized Braid Lemma \ref{GeneralizedBraid} to
$J=g_{(p,r)},K=\delta
_{(p,r)}^{(d+1)}(.,.,W_{3},...,W_{d+1})$, $J^{\prime }=-\partial _{(p,r)}g$ and $%
A=\phi _{(p,r)}^{(d+2) }(.,.,.,W_{3},...,W_{d+1})$,  we get $\phi
_{(p,r)}^{(d+2)}=0$ and
\begin{equation*}
-\delta _{(p,r)}^{(d+1)}(W_{1},...,W_{d+1})\partial _{(p,r)}g(U,V)=0
\end{equation*}%
which means that $\delta _{(p,r)}^{(d+1)}=0$ since $g$ is nowhere
transversally Riemannian. Therefore $\Psi $ has a trivial $(d+1)$-jet. This
completes the proof of Theorem \ref{subrigid}.

\end{proof}

\end{document}